\documentclass[12pt]{article}

\usepackage{amsthm}
\usepackage{amssymb}
\usepackage{amsmath}
\usepackage{amsfonts}
\usepackage{times}

\usepackage[paper=a4paper, top=2.5cm, bottom=2.5cm, left=2.5cm, right=2.5cm]{geometry}

\def\qed{\hfill $\vcenter{\hrule height .3mm
\hbox {\vrule width .3mm height 2.1mm \kern 2mm \vrule width .3mm
height 2.1mm} \hrule height .3mm}$ \bigskip}

\def \Sph{\mathbb{S}^{n-1}}
\def \RR {\mathbb R}

\def \EE {\mathbb E}

\def \Var {\mathrm{Var}}

\def \PP {\mathbb P}
\def \eps {\varepsilon}

\def \DC {\mathcal{C}_n}

\def \Id {\mathbf{I}_n}
\def \COV {\mathrm{Cov}}

\def \limsup {\mathrm{limsup}}

\def \TTT {\mathcal{T}}

\def \FF {\mathcal{F}}

\newtheorem{theorem}{Theorem}
\newtheorem{lemma}[theorem]{Lemma}

\newtheorem{fact}[theorem]{Fact}

\newtheorem{proposition}[theorem]{Proposition}
\newtheorem{corollary}[theorem]{Corollary}
\theoremstyle{definition}
\newtheorem{definition}[theorem]{Definition}
\theoremstyle{remark}
\newtheorem{remark}[theorem]{Remark}

\long\def\symbolfootnotetext[#1]#2{\begingroup
\def\thefootnote{\fnsymbol{footnote}}\footnotetext[#1]{#2}\endgroup}

\author{Ronen Eldan\thanks{Weizmann Institute of Science. Supported by a European Research Council Starting Grant (ERC StG) and by an Israel Science Foundation grant no. 715/16.} ~and Omer Shamir\thanks{Weizmann Institute of Science. Supported by the European Research Council.}}

\title{Log concavity and concentration of Lipschitz functions on the Boolean hypercube}

\begin{document}
\maketitle

\begin{abstract}
It is well-known that measures whose density is the form $e^{-V}$ where $V$ is a uniformly convex potential on $\RR^n$ attain strong concentration properties. In search of a notion of log-concavity on the discrete hypercube, we consider measures on $\{-1,1\}^n$ whose multi-linear extension $f$ satisfies $\log \nabla^2 f(x) \preceq \beta \Id$, for $\beta \geq 0$, which we refer to as $\beta$-semi-log-concave. We prove that these measures satisfy a nontrivial concentration bound, namely, any Hamming Lipchitz test function $\varphi$ satisfies $\Var_\nu[\varphi] \leq n^{2-C_\beta}$ for $C_\beta>0$. As a corollary, we prove a concentration bound for measures which exhibit the so-called Rayleigh property. Namely, we show that for measures such that under any external field (or exponential tilt), the correlation between any two coordinates is non-positive, Hamming-Lipschitz functions admit nontrivial concentration.
\end{abstract}

\section{Introduction}
In the Euclidean space $\RR^n$, \emph{log-concave} measures, namely measures whose density is of the form $e^{-U}$ with $U$ being a convex potential are known to satisfy several concentration inequalities. For instance, if $\gamma$ is the standard Gaussian measure on $\RR^n$ and the probability measure $\nu$ is absolutely continuous with respect to $\gamma$ with density $d \nu = e^{V} d \gamma$ where the potential $V:\RR^n \to \RR$ satisfies the condition
\begin{equation}\label{eq:Rnsemilogconcave}
\nabla^2 V \preceq (1-\delta) \Id,
\end{equation}
then for every $1$-Lipschitz test function $\varphi$, we will have $\Var_\nu[\varphi] \leq \frac{1}{\delta}$, see e.g., \cite{Saumard14}. In fact, much stronger concentration, for instance in the form of a logarithmic Sobolev inequality, is known to hold in this case.

The objective of this work is to try to generalize the notion of log-concavity to the Boolean hypercube in a way that analogous concentration inequalities are attained. Define $\DC := \{-1,1\}^n$. We say that a function $\varphi:\DC \to \RR$ is $1$-(Hamming)-Lipschitz if
$$
|\varphi(x) - \varphi(y)| \leq \|x-y\|_1, ~~ \forall x,y \in \DC.
$$
Let $\mu$ be the uniform measure on $\DC$. Suppose that $d \nu = e^V d \mu$. Given a $1$-Lipschitz function $\varphi$, it is a trivial fact that $\Var_\nu[\varphi] \leq n^2$, and this bound is sharp in general (e.g., take $\varphi(x) = \sum_i x_i$ and $\nu$ which assigns mass $1/2$ to $(-1,\cdots,-1)$ and $(1,\dots,1)$). We are interested in the question of finding sufficient conditions on the potential $V$, analogous to \eqref{eq:Rnsemilogconcave}, under which a nontrivial bound for $\Var_\nu[\varphi]$ is implied. Let us mention that in the case of the \emph{continuous} hypercube, concentration results of this nature we obtained by Klartag (\cite{Klartag-cube}).

It is clear that any potential $V: \DC \to \RR$ is the restriction of some convex function on $\RR^n$ to $\DC$, meaning that the notion of convexity has to either consider the discrete derivatives of $V$ or to consider the continuous Hessian applied to a suitably chosen interpolation. Our suggested notion of convexity is roughly based on the multi-linear interpolation, but the formal definition will first be given in terms of the logarithmic Laplace transform. Define,
$$
\mathcal{L}[\nu](x) = \log \int_{\DC} e^{\langle x, y \rangle} d \nu(y), ~~ \forall x \in \RR^n.
$$
The function $\mathcal{L}[\nu]$ is known as the log-Laplace transform of the measure $\nu$. We are now ready to define our main notion of \emph{semi-log-concavity}.
\begin{definition} (Semi log-concave measures).
Given a measure $\nu$ on $\DC$, We say that $\nu$ is $\beta$-semi-log-concave if
\begin{equation}\label{eq:semilc}
\nabla^2 \mathcal{L}[\nu](x) \preceq \beta \Id, ~~ \forall x \in \RR^n,
\end{equation}
where the inequality is in the positive-definite sense.
\end{definition}

Our main theorem gives a nontrivial concentration bound for Lipschitz functions with respect to such measures.
\begin{theorem} \label{thm:main}
	Let $\beta \geq 1$. If $\nu$ is a $\beta$-semi-log-concave probability measure on $\DC$ and $\varphi$ is $1$-Lipschitz, then
	$$
	\Var_\nu[\varphi] \leq C \beta n^{2 - c/\beta} 
	$$
	where $C,c>0$ are universal constants.
\end{theorem}
The theorem shows in particular that for any $\beta \in \RR$ there exists $n$ large enough such that $\beta$-semi-log-concave measures on $\DC$ admit nontrivial concentration. \\

Before we proceed, let us give an alternative and slightly stronger notion of semi-log-concavity which could hopefully shed some light and give better intuition regarding the relation to the usual notion of log concavity in $\RR^n$. We first need to recall the \emph{multi-linear} extension of the measure into the continuous hypercube, $[-1,1]^n$. Given a function $f:\DC \to \RR$, it is known that there is a unique function $\hat f:2^{[n]} \to \RR$ such that
\begin{equation}\label{eq:Fourier}
f(x) = \sum_{A \subset [n]} \hat f(A) \prod_{i \in A} x_i.
\end{equation}
The function $\hat f$ is known as the Walsh-Fourier transform of $f$. Observing that the above form makes sense as a function from $[-1,1]^n$ to $\RR$, we refer to this as the multi-linear (or harmonic) extension of $f$. Since the two coincide on $\DC$, below we allow ourselves to use the same notation for both the function and its multi-linear extension. The following fact is attained via a simple calculation, see Section \ref{sec:harmonic} below.
\begin{fact} \label{fact:harmonic}
Suppose that 
\begin{equation}\label{eq:lc2}
\nabla^2 \log \frac{d \nu}{d \mu}(x) \preceq \beta \Id, ~~ \forall x \in [-1,1]^n.
\end{equation}
Then $\nu$ is $(\beta+3)$-semi-log-concave.
\end{fact}
Thus, a corollary to Theorem \ref{thm:main} is that the condition \eqref{eq:lc2} implies that every $1$-Lipschitz function $\varphi$ satisfies $\Var_\nu[\varphi] \leq C(\beta+3)n^{2 - c/(\beta+3)}$.
\medskip
\begin{remark}
Let $f$ be the multi-linear extension of $\frac{d \nu}{d \mu}$. Since $f$ is harmonic and since $\nabla \log f = \frac{\nabla^2 f}{f} - \frac{\nabla f^{\otimes 2}}{f^2}$, we see that in order for $f$ to be log-concave, the Hessian of $f$ can have at most one positive eigenvalue, due to which the family of log-concave functions, or in other words functions satisfying condition \eqref{eq:lc2} with $\beta = 0$, is rather restricted. Semi-log-concave measures are a much richer family, as demonstrated below. 
\end{remark}

\subsection{A lower bound on the entropy}
Our second result addresses the question of finding conditions under which the entropy of the measure $\nu$ is close to that of the corresponding product measure, having the same marginals as $\nu$. This result is inspired by a corresponding bound due to Anari, Oveis-Gharan and Vinzant \cite[Theorem 5.2]{AOV1}, see the discussion below.

Given a measure $\nu$ on $\DC$, we define,
$$
\mathcal{H}(\nu) := \int_{\DC} \log \frac{1}{\nu(\{y\})} d\nu(y),
$$
the entropy of $\nu$. Moreover, for all $i$ let $\pi_i(\nu)$ be the marginal of $\nu$ onto the $i$-th coordinate, and define
$$
\tilde {\mathcal{H}}(\nu) := \sum_{i \in [n]} \mathcal{H}(\pi_i(\nu)).
$$
A well-known fact is that $\mathcal{H}(\nu) \leq \tilde {\mathcal{H}}(\nu)$. We show that under a log-concavity-type condition, this inequality can be reversed. 
\begin{theorem} \label{thm:entropy}
Let $\nu$ be a probability measure on $\DC$. Suppose that, for some $\beta \geq 1$, $\nu$ satisfies the condition
\begin{equation}\label{eq:semilc2}
\nabla^2 \mathcal{L}[\nu](x) \preceq \beta \mathrm{diag}\left ( \nabla^2 \mathcal{L}[\nu](x) \right) , ~~ \forall x \in \RR^n.
\end{equation}
Then one has
$$
\tilde {\mathcal{H}}(\nu) \leq \beta \mathcal{H}(\nu).
$$
\end{theorem}
\begin{remark}
Condition \eqref{eq:semilc2} is stronger than condition \eqref{eq:semilc}. Indeed, it is not hard to check that $\mathrm{diag}\left ( \nabla^2 \mathcal{L}[\nu](x) \right) \preceq \Id$ (see formula \eqref{eq:momentgenerating} below).
\end{remark}

\subsection{An application: Concentration of negatively dependent random variables}
For a sequence of Bernoulli variables $X_1,...,X_n$, there are several notions of \emph{negative-dependence} between those variables (see e.g., \cite{BSS82,BK85,BF87,DR98}), some of which are known or conjectured to imply concentration of Lipschitz functions. Some notable first steps towards a theory unifying those notions appear in the work of Permantle \cite{pemantle04}, to which we also refer for a review of these notions.

The simplest notion of negative dependence is \textbf{pairwise negative-correlations}, hence the condition that $\EE[X_i X_j] \leq \EE X_i \EE X_j$ for all $i \neq j$. This condition, however, is too weak to imply any nontrivial concentration bounds. For example if $(X_1,...,X_n)$ are distributed as a uniformly chosen row of the $n \times n$ Hadamard matrix, then these variables are pairwise independent, and therefore have nonpositive correlations. However, if $A$ is the subset of combinations given by the first $n/2$ rows and $\varphi(x)$ is the Hamming distance of $x$ to the set $A$, then $\Var[ \varphi(X_1,\dots,X_n)] = \Omega \left (n^2 \right )$. 

A stronger notion which appears in the literature is \textbf{negative-association}: We say that $X_1,...,X_n$ are negatively associated if for all $I,J \subset [n]$ with $I \cap J = \emptyset$ and every monotone functions $f:\{0,1\}^I \to \RR$ and $g:\{0,1\}^J \to \RR$ one has $\EE [f(X_I) g(X_J)] \leq \EE f(X_I) \EE g(X_J)$. It was conjectured by E. Mossel that Lipschitz functions admit sub-Gaussian concentration with respect to such measures. 

To the best of our knowledge, there are two results in this direction in the literature: It was shown by Peres and Pemantle that sub-Gaussian concentration of Lipschitz functions hold for measures satisfying the \textbf{strong-Rayleigh} property (\cite{Peres-Pemantle}), which amounts to stability of the generating polynomial of the measure. More recently, Garbe and Vondrak (\cite{Garbe-Vondrak}) showed that concentration is implied by the \textbf{negative-regression} property. We also refer to their paper for a discussion of related bounds and questions. \\

Here, we consider the following notion of negative dependence suggested by Wagner \cite{Wa08}.
\begin{definition} (Rayleigh measures). 
We say that $X_1,\dots,X_n$ satisfy the \textbf{Rayleigh} property if for every $\theta \in \RR^n$ and for all $i,j \in [n]$, we have
\begin{equation}\label{eq:TNC}
\EE \left [X_i X_j e^{ \sum_i \theta_i X_i} \right ] \EE \left [ e^{ \sum_i \theta_i X_i} \right ] \leq \EE\left [X_i e^{ \sum_i \theta_i X_i} \right ] \EE\left [ X_j e^{ \sum_i \theta_i X_i} \right ]. 
\end{equation}
\end{definition} 

In other words, $X_1,...,X_n$ satisfies the Rayleigh property if the correlations between all pairs are negative even after reweighing the measure by an exponential tilt. Equivalently, this is the largest family that exhibits pairwise negative correlations and is closed under the operation of applying a magnetic field.\\

A corollary of our main theorem is the following concentration bound for measures with the Rayleigh property as well a lower bound for the entropy.
\begin{corollary} \label{cor:TNC}
If $X_1,...,X_n$ satisfy the Rayleigh property, then,
\begin{enumerate}
\item 
For any $1$-Hamming-Lipschitz function $\varphi$, we have 
$$
\Var[\varphi(X_1,\dots,X_n)] \leq C n^{2-c},
$$ 
for universal constants $C,c>0$.
\item 
One has,
$$
\sum_{i \in [n]} \mathcal{H}(X_i) \leq 2 \mathcal{H}(X_1,...,X_n).
$$	
\end{enumerate}
\end{corollary} 

\begin{proof}[Proof of corollary \ref{cor:TNC}]
Let $\nu$ be the law of $(2 X_1 - 1,\dots, 2 X_n - 1)$. Observe that the condition \eqref{eq:TNC} is equivalent to 
$$
\partial_i \partial_j \mathcal{L}[\nu](\theta) \leq 0,\quad\quad\forall{i\neq j}.
$$
Therefore, Rayleigh property is equivalent to the fact that $\nabla^2 \mathcal{L}[\nu](x)$ has non-positive off-diagonal entries for all $x \in \RR^n$. Let $u \in \RR^n$. Define $u = u_+ + u_-$ where $u_+ \in \RR_+^n$ and $u_- \in \RR_-^n$. Recall that $\nabla^2 \mathcal{L}[\nu](x)$ is positive semi-definite and that $\left (\nabla^2 \mathcal{L}[\nu](x) \right )_{i,i} \leq 1$ for all $i$ (see the identity \eqref{eq:momentgenerating} below), so by convexity,
\begin{align*}
\left \langle u, \nabla^2 \mathcal{L}[\nu](x) u \right \rangle & \leq 2 \left \langle u_+, \nabla^2 \mathcal{L}[\nu](x) u_+ \right \rangle + 2 \left \langle u_-, \nabla^2 \mathcal{L}[\nu](x) u_- \right \rangle \\
& \leq 2 \langle u_+, \mathrm{diag}\left ( \nabla^2 \mathcal{L}[\nu](x) \right ) u_+ \rangle + 2 \langle u_-, \mathrm{diag}\left ( \nabla^2 \mathcal{L}[\nu](x) \right ) u_- \rangle \\
& = 2 \langle u, \mathrm{diag}\left ( \nabla^2 \mathcal{L}[\nu](x) \right ) u \rangle.
\end{align*}
Thus, $\nu$ satisfies both \eqref{eq:semilc} and \eqref{eq:semilc2} with $\beta = 2$. An application of Theorem \ref{thm:main} implies the first part, and an application of Theorem \ref{thm:entropy}, the second.
\end{proof}

It was pointed out to us by P. Nuti and J. Vondrak, that in the special case that the measure $\nu$ is homogeneous (namely when $\sum_i X_i$ is deterministic), the Rayleigh property implies both the so-called stochastic covering property (see \cite{Peres-Pemantle}) and the negative regression property (\cite{Garbe-Vondrak}), which in turn (using either of the above references) implies a stronger version the above corollary (which gives sub-Gaussian concentration). In fact, the more recent paper \cite{anari2020spectral} gives spectral gap in this case (see discussion below). However, all of the above seem to rely on homogeneouity in a crucial way.


\subsection{Relation to the works of Anari, Liu, Oveis-Gharan and Vinzant}
A seemingly related notion of log-concavity of measures on the discrete hypercube was given in a series of works by Anari, Liu, Oveis-Gharan and Vinzant in \cite{AOV2,AOV1}. Given $\{-1,1\}$-Bernoulli random variables $X_1,\dots,X_n$ distributed according to a law $\nu$, which can be identified with a random subset $A \subset [n]$ by $X_i = 2 \mathbf{1}_{i \in A} - 1$, they consider the generating polynomial
$$
p_\nu(z_1,...,z_n) = \EE \prod_{i \in A} z_i.
$$
They show that if $p$ is both log-concave on the positive orthant and homogeneous (which is equivalent to the fact $\sum_i X_i$ is supported on one point), then the law of $X_1,...X_n$ admits, among other things, strong concentration properties in the form of a spectral gap (with respect to the Glauber dynamics). 

It is not hard to check that the log-concavity of the polynomial $p_\nu$ is equivalent to the condition
\begin{equation}\label{eq:condaov}
\nabla^2 \mathcal{L}[\nu] (x) \preceq 2 \left (\mathrm{diag} \left (\nabla \mathcal{L}[\nu] (x) \right ) + \Id \right ), ~~ \forall x \in \RR^n.
\end{equation}
Since $\nabla \mathcal{L}[\nu] (x) \in [-1,1]^n$, the above condition is strictly stronger our semi-log-concavity condition \eqref{eq:semilc} with $\beta=4$. 

On a first glance it may seem that Theorem \ref{thm:main} is effectively similar to \cite[Theorem 1.1]{AOV2} (and could perhaps follow from the same methods), however we believe that this is not the case, and the resemblance between the two results is mainly on a superficial level. A crucial difference between the results is that our notion of log-concavity is invariant under reflections about the coordinate axes, whereas in the latter notion, the direction $(1,...,1)$ has a special role. In cases of interest, such as homogeneous distributions where $|A| \ll n$, condition \eqref{eq:condaov} is actually closer to strict log-concavity. 

Since the multi-linear polynomial $p_\nu$ is harmonic, its log-concavity implies that the Hessian matrix can have only one non-negative eigenvalue. In this sense, condition \eqref{eq:condaov} is much more rigid than our condition. Respectively, while our proof is based mainly on analytic methods, the proof in \cite{AOV2} has a more algebraic flavor (and is also based the theory of high-dimensional expanders). We do not know if there is a deeper connection between the results, but it doesn't seem that any of the two follows from the other. 

On the other hand, Theorem \ref{thm:entropy} seems rather closely related to \cite[Theorem 5.2]{AOV1}, and the former can be thought of as a soft and modified version of the latter: Indeed, condition \eqref{eq:semilc2}, compared to \eqref{eq:condaov} is invariant under coordinate reflections and softer in the sense that $\beta$ can be larger than $1$, but otherwise rather similar. The proof of the latter is simpler and basically reduces to an application of Jensen's inequality, whereas our proof (a small variation thereof also implies the latter bound) is slightly more complicated and uses stochastic calculus; we do not know if it can be attained by more elementary techniques.

Finally, in relation to Corollary \ref{cor:TNC}, a result of a similar spirit appears in \cite{anari2020spectral}. A corollary of the main theorem there shows that if $\nu$ is $d$-homogeneous and all measures obtainable from $\nu$ by conditioning are pairwise-negatively correlated, then it has a spectral gap polynomial in $n$. It seems however, that the assumption of homogeneouity is crucial in this case. 

\subsubsection*{Acknowledgements}
We'd like to thank Bo'az Klartag for a fruitful discussion, as well as Nima Anari and Jan Vondrak and Pranav Nuti for some enlightening comments on a preliminary version of this manuscript.

\section{Preliminaries and a stochastic construction}
Throughout this section, we fix a probability measure $\nu$ on $\DC$ and a Lipschitz test function $\varphi:\DC \to \RR$.
\subsection{Some preliminary definitions}
For a vector $w \in \RR^n$, define the \emph{tilt} of the measure $\nu$ as
$$
\frac{d \tau_{w} \nu(x)}{d \nu(x)} := Z_\nu(w)^{-1}  e^{\langle w, x \rangle} 
$$
where
$$
Z_\nu(w) := \int_{\DC} e^{\langle w, x \rangle} d \nu.
$$
Also define the functions
$$
a_\nu(w) := \int_{\DC} x d \tau_w \nu(x), ~~ A_\nu(w) := \int_{\DC} \left (x - a_\nu(w) \right )^{\otimes 2} d \tau_w \nu(x) = \COV(\tau_w \nu).
$$
A well-known calculation gives,
\begin{equation}\label{eq:momentgenerating}
a_\nu(w) = \nabla \mathcal{L} [\nu] (w), ~~ A_\nu(w) = \nabla^2 \mathcal{L} [\nu] (w)
\end{equation}
(this is the fact that the Log-Laplace transform is the cumulant-generating function). Thus, if $\nu$ is $\beta$-log-concave, we have
\begin{equation}\label{eq:nablaa}
\nabla a_\nu(w) = \nabla^2 \mathcal{L}[\nu] (w) \preceq \beta \Id.
\end{equation}
\subsection{Stochastic localization} \label{sec:stochastic}
We construct a stochastic process driven by a Brownian motion, which we refer to as \emph{stochastic localization}. A somewhat similar process was originally used in \cite{Eldan-thin-KLS} to establish concentration properties for log-concave measures on $\RR^n$. Here we use a discrete version, similar to the construction which appears in \cite{Eldan-taming, eldan2020spectral}. In this section we occasionally allow ourselves to omit some of the details of the proofs, and the reader is referred to \cite{Eldan-taming} for more rigorous derivations. \\

Let $B_t$ be a standard Brownian motion on $\RR^n$ adapted to a filtration $\FF_t$. Consider the system of equations,
\begin{equation}\label{eq:localization}
F_0(x) = 0, ~~ d F_t (x) = F_t(x) \langle x - a_t,  d B_t \rangle, ~~ \forall x \in \DC,
\end{equation}
where $a_t := \int x d \nu_t(x) := \int x F_t(x) d \nu(x)$. \\

We think of this process $(\nu_t)_t$ as an evolution of measures on $\DC$, which starts with the measure $\nu_0 = \nu$, and as seen below, ends up with a Dirac measure whose support is $\nu$-distributed. Let us first summarize some useful properties of this process.
\begin{proposition} \label{prop:basicprops}
The process defined above satisfies the following properties.
\begin{enumerate}
\item 
Almost surely, for all $t$, $\nu_t$ is a probability measure.
\item
For all $A \subset \DC$, the process $\nu_t(A)$ is a martingale.
\item 
The process $a_t$ almost surely converges to a point in $\DC$, and $a_\infty := \lim_{t \to \infty} a_t$ is distributed according to the law $\nu$. Moreover, the measure $\nu_t$ almost-surely weakly converges to a Dirac measure at $a_\infty$.
\end{enumerate}
\end{proposition}
\begin{proof}
We have
$$
d \nu_t(\DC) = d \int_{\DC} F_t(x) \nu(x) = \int_{\DC} (x - a_t) d \nu_t(x) d B_t = 0,
$$
which proves the first part. The second part is evident from the definition. For the third part, a calculation gives,
\begin{align}
da_t ~& = d \int_{\DC} x \nu_t(x)  \nonumber \\
& = \left (\int_{\DC} x \otimes (x - a_t)\nu_t(dx) \right ) dB_t \nonumber \\
& = \left (\int_{\RR^d} (x - a_t)^{\otimes 2} \nu_t(dx) \right ) dB_t \nonumber \\ 
& = \COV(\nu_t) dB_t. \label{eq:dat}
\end{align}
Therefore, $a_t$ is a martingale, and 
\begin{equation}\label{eq:atfast}
d [\langle a_t, e_i \rangle]_t = \sum_{j \in [n]} \COV(\nu_t)_{i,j}^2 \geq \COV(\nu_t)_{i,i}^2 = (1 - \langle a_t, e_i \rangle^2)^2 dt,
\end{equation}
implying that $\langle a_t, e_i \rangle$ converges to $\{\pm 1\}$ almost surely and that $\nu_t$ converges weakly to to $\delta_{a_\infty}$, which also implies that
$$
\lim_{t \to \infty} a_t \in A \Leftrightarrow \lim_{t \to \infty} \nu_t(A) = 1,
$$
for all $A \subset \DC$. Since $\nu_t(A)$ is a martingale, we have that 
$$
\PP \left ( \lim_{t \to \infty} a_t \in A \right ) = \lim_{t \to \infty} \EE \nu_t(A) = \nu(A),
$$
implying the third part.
\end{proof}

Next, we have by It\^o's formula, for all $x \in \DC$,
\begin{align}
d \log F_t(x) & = \frac{d F_t(x)}{F_t(x)} - \frac{d[F(x)]_t}{2 F_t(x)^2} \nonumber \\
& = \langle x-a_t, d B_t \rangle - \frac{1}{2}  |x-a_t|^2 dt \nonumber \\
& = \langle x, d B_t + a_t dt \rangle + d Z_t
\end{align}
where $Z_t$ is an It\^o process that does not depend on $x$ (here we used the fact that $|x|^2$ is constant on $\DC$). Therefore,
$$
\log F_t(x) = \left \langle x, w_t  \right \rangle + c_t
$$
where $c_t$ is some It\^o process and $w_t = B_t + \int_0^t a_s ds$. The above display and the fact that $\nu_t$ is a probability measure, implies that
\begin{equation}\label{eq:tilts}
\nu_t = \tau_{w_t} \nu,
\end{equation}
and therefore also 
$$
a_t = \int_{\DC} x d \nu_t(x) = \int_{\DC} x d \tau_{w_t} \nu(x) = a_\nu(w_t).
$$ 
The process $w_t$ thus satisfies the equation
\begin{equation}\label{eq:dvt}
w_0 = 0, ~~ d w_t = d B_t + a_\nu(w_t) dt.
\end{equation}
The above equation gives an alternative construction for the process defined in \eqref{eq:localization}. Due to the Markov property of the measure-valued process $\nu_t$, the above discussion leads to the following result.
\begin{proposition} \label{prop:alt}
Given a measure $\nu$ and a vector $v \in \RR^n$, consider the process defined by the equation
$$
u_0 = v, ~~~ d u_t = d B_t + a_\nu(u_t) dt.
$$
Then $X = \lim_{t \to \infty} a_\nu(u_t)$ exists and is a point in $\DC$ almost surely, and has the law $\tau_{v} \nu$. 
\end{proposition}
\begin{proof}
Consider the measure $\tilde \nu = \tau_{v} \nu$ and let $w_t$ be the process constructed as above with $\nu$ replaced by $\tilde \nu$ in equation \eqref{eq:localization}. Observe that by definition 
$$
a_{\tilde \nu}(x) = a_\nu(x + v), ~~ \forall x \in \RR^n.
$$	
In light of \eqref{eq:dvt} and by the uniqueness of the solution to the above SDE, we have that $w_t + v = u_t$ almost surely, for all $t$. The result now follows from part 3 of Proposition \ref{prop:basicprops} and the fact that 
$$
a_t = a_{\tilde \nu}(w_t) = a_\nu(u_t).
$$
\end{proof}

\section{Proof of Theorem \ref{thm:main}}
\subsection{Estimating the variance in terms of the transportation distance between tilts}
Define $M_t = \int \varphi d \nu_t$. Observe that by part 2 of Proposition \ref{prop:basicprops}, $M_t$ is a martingale. We first claim that
\begin{equation}\label{eq:totalvar}
\Var_\nu[\varphi] = \EE [M]_t + \EE \Var_{\nu_t}[\varphi],
\end{equation}
where $[M]_t$ denotes the quadratic variation of $M_t$.

Indeed, by part 3 of Proposition \ref{prop:basicprops}, we have that $M_\infty := \lim_{t \to \infty} M_t$ exists almost surely and has the law $\varphi_\star \nu$. Consequently, $\Var_\nu[\varphi | \FF_t] = \Var[M_\infty | \FF_t]$ almost surely, for all $t$. By It\^o's isometry, we have
$$
\Var_\nu[\varphi] = \EE [M]_t + \EE \Var[M_\infty | \FF_t] = \EE [M]_t + \EE \Var_{\nu_t}[\varphi].
$$
We will carry on by bounding each of the terms on the right hand side separately, for a suitable chosen value of $t$. We begin with the first term, for which we calculate
$$
d M_t = d \int_{\DC} \varphi(x) \nu_t(x) = \int_{\DC} \varphi(x) d F_t(x) \nu(x) = \int_{\DC} \varphi(x) \langle x-a_t, d B_t \rangle \nu_t(x).
$$
Therefore, we can estimate
\begin{align}
d [M]_t & = \left | \int_{\DC} \varphi(x)  (x-a_t) d \nu_t  \right |^2 dt \nonumber \\
& \leq \sup_{|\theta|=1, \tilde \varphi \in \mathrm{Lip}(\DC)}  \left | \int_{\DC} \tilde \varphi(x)  \langle x-a_t, \theta \rangle d \nu_t \right |^2 dt \nonumber \\
& = \sup_{|\theta|=1, \tilde \varphi \in \mathrm{Lip}(\DC)}  \left | \int_{\DC} \left (\tilde \varphi(x) -\int_{\DC} \tilde \varphi d \nu_t \right )  \langle x-a_t, \theta \rangle d \nu_t \right |^2 dt \nonumber \\
& = \sup_{|\theta|=1, \tilde \varphi \in \mathrm{Lip}(\DC)}  \left | \int_{\DC} \left (\tilde \varphi(x) -\int_{\DC} \tilde \varphi d \nu_t \right )  \limsup_{\eps \to 0+} \frac{1}{\eps}   \left (\exp \left (\langle x-a_t, \eps \theta \rangle \right ) - 1 \right ) d \nu_t \right |^2 dt \nonumber \\
& = \sup_{|\theta|=1, \tilde \varphi \in \mathrm{Lip}(\DC)}  \left | \limsup_{\eps \to 0+} \frac{1}{\eps} \int_{\DC} \left (\tilde \varphi(x) -\int_{\DC} \tilde \varphi d \nu_t \right ) \exp \left (\langle x, \eps \theta \rangle \right ) d \nu_t \right |^2 dt \nonumber \\
& = \sup_{|\theta|=1, \tilde \varphi \in \mathrm{Lip}(\DC)}  \left | \limsup_{\eps \to 0+} \frac{1}{\eps} Z_{\nu_t} (\eps \theta) \int_{\DC} \left (\tilde \varphi(x) - \int_{\DC} \tilde \varphi d \nu_t \right ) d \tau_{\eps \theta} \nu_t \right |^2 dt \nonumber \\
& = \sup_{|\theta|=1, \tilde \varphi \in \mathrm{Lip}(\DC)}  \left | Z_{\nu_t} (0) \limsup_{\eps \to 0+} \frac{1}{\eps} \left (\int_{\DC} \tilde \varphi d \tau_{\eps \theta} \nu_t - \int_{\DC} \tilde \varphi d \nu_t \right )  \right |^2 dt \nonumber \\ 
& = \sup_{|\theta|=1} \left |\limsup_{\eps \to 0+} \frac{1}{\eps} \mathrm{W}_1 ( \nu, \tau_{\eps \theta} \nu ) \right |^2 dt, \label{eq:wass}
\end{align}
where for two measures $\nu,\tilde \nu$, we define
$$
\mathrm{W}_1(\nu,\tilde \nu) = \sup_{\tilde \varphi \in \mathrm{Lip}(\DC)} \Bigl |\EE_{\nu}[\tilde \varphi] - \EE_{\tilde \nu}[\tilde \varphi] \Bigr |
$$
known as the Wasserstein transportation distance between $\nu$ and $\tilde \nu$. \\

Towards bounding the second term of the right hand side of \eqref{eq:totalvar}, define $A_t = \COV(\nu_t)$.
\begin{fact} \label{fact:smalltail}
One has,
$$
Var_{\nu_t}[\varphi] \leq n \mathrm{Tr}(A_t).
$$
\end{fact}
\begin{proof}
It is easily checked that $\varphi$ is $1$-Hamming-Lipschitz then its multi-linear extension satisfies $|\partial_i \varphi(x)| \leq 1$ for all $x \in [-1,1]^n$ and $i \in [n]$. Therefore, its multi-linear extension is $\sqrt{n}$-Lipschitz with respect to the Euclidean distance, hence
$$
\left | \varphi(x) - \varphi(y) \right | \leq \sqrt{n} |x-y|, ~~ \forall x,y \in [-1,1]^n.
$$
So, if $X \sim \nu_t$ then,
$$
Var_{\nu_t}[\varphi] \leq \EE [ \left ( \varphi(X) - \varphi(\EE[X]))^2 \right ] \leq n \EE \left [|X - \EE[X]|^2\right ] = n \mathrm{Tr}(A_t).
$$
\end{proof}

\begin{lemma}
For all $t,r \geq 0$, we have almost surely,
\begin{equation}\label{eq:Trdecreasing}
\EE[\mathrm{Tr}(A_{t+r}) | \FF_r] \leq n e^{-t/8}.
\end{equation}
\end{lemma}
\begin{proof}
Fix $i \in [n]$, and define 
$$
S_t := (A_t)_{i,i}, ~~ Q_t = \langle a_t, e_i \rangle.
$$ 
By part 3 of Proposition \ref{prop:basicprops}, we have that $a_\infty | \FF_t$ has the law $\nu_t$, meaning that $Q_\infty := \lim_{t \to \infty} Q_t$ exists almost surely, and that
$$
S_t = \Var[Q_\infty | \FF_t].
$$
Recall that $a_t$ is a martingale and thus so is $Q_t$ and
$$
S_t = \Var[Q_\infty | \FF_t] = \EE[Q_\infty^2 | \FF_t] - Q_t^2 = 1 - Q_t^2.
$$
Equation \eqref{eq:atfast} can be written
\begin{equation}\label{eq:Qtfast}
d [Q]_t \geq S_t^2 dt.
\end{equation}
By It\^o's formula,
$$
d S_t = - 2 Q_t d Q_t - d [Q]_t,
$$
and
$$
d \sqrt{S_t} = \frac{ - 2 Q_t d Q_t - d [Q]_t }{2 \sqrt{S_t}} - \frac{1}{8} \frac{d [Q]_t}{ S_t^{3/2} } \stackrel{\eqref{eq:Qtfast}}{\leq } \frac{ - Q_t d Q_t }{ \sqrt{S_t}} - \frac{\sqrt{S_t}}{8} dt.
$$
Since the first term on the right hand side is a martingale, we have almost surely, for all $s,t \geq 0$,
$$
\frac{d}{dt} \EE[ \sqrt{S_{t+r}} | \FF_r ] \leq - \frac{1}{8} \EE[ \sqrt{S_{t+r}} | \FF_r ].
$$
By integrating (using Growall's inequality), we finally get
$$
\EE [ S_{t+r} | \FF_r ] \leq \EE [ \sqrt{S_{t+r}} | \FF_r ] \leq \sqrt{S_r} e^{-t/8} \leq e^{-t/8}
$$
(where we used the fact that $S_t \leq 1$ almost surely for all $t$). The proof is completed by summing over coordinates.
\end{proof}
Finally, combining equations \eqref{eq:totalvar}, \eqref{eq:wass}, \eqref{eq:Trdecreasing} and Fact \ref{fact:smalltail}, we have for all $T>0$, 
\begin{align}
\Var[\varphi] & \leq \EE \int_0^T  \sup_{|\theta|=1}  \left | \limsup_{\eps \to 0+} \frac{1}{\eps} \mathrm{W}_1(\nu_t, \tau_{\eps \theta} \nu_t) \right |^2 dt + n^2 e^{-T/8} \nonumber \\
& \stackrel{\eqref{eq:tilts}}{=} \EE \int_0^T  \sup_{|\theta|=1}  \left | \limsup_{\eps \to 0+} \frac{1}{\eps} \mathrm{W}_1(\tau_{w_t} \nu, \tau_{w_t + \eps \theta} \nu) \right |^2 dt + n^2 e^{-T/8}.  \label{eq:finalvarbound}
\end{align}

In light of the above bound, the proof boils down to the following estimate on the transportation distance between two close tilts.
\begin{proposition} \label{prop:transport}
Let $\nu$ be a $\beta$-semi-log-concave measure on $\DC$. Let $v \in \RR^n$ and let $\theta \in \mathbb{S}^{n-1}$. Then, for all $0 < \eps < 0.1$, we have
$$
\mathrm{W}_1(\tau_v \nu, \tau_{v + \eps \theta} \nu) \leq 4 \eps \beta n^{1-1/(32 \beta)}.
$$
\end{proposition}
\begin{proof}[Proof of Theorem \ref{thm:main}]
Use equation \eqref{eq:finalvarbound} with $T = 16 \log n$. Invoke the above proposition and attain
$$
\Var[\varphi] \leq 4 T \beta n^{2-1/(16 \beta)} + n^2 e^{-2 \log n} \leq C \beta n^{2-1/(17 \beta)},
$$
for a universal constant $C>0$.
\end{proof}

\subsection{Proof of Proposition \ref{prop:transport}: The stochastic coupling}
Proposition \ref{prop:transport} will be proven via a coupling argument laid out below. Let $v \in \RR^n$ and consider the process
\begin{equation}\label{eq:dvtc}
w_0 = v, ~~~~ d w_t = d B_t + a_\nu(w_t) dt.
\end{equation}
According to Proposition \ref{prop:alt}, we have that $\lim_{t \to \infty} a_\nu(w_t) \sim \tau_{v} \nu$. This gives rise to the following coupling. Let $U_t$ be process adapted to $\FF_t$ such that for all $t$, $U_t$ is an orthogonal matrix. Let $\eps > 0$ and $\theta \in \Sph$, and consider the additional process defined by the equation
\begin{equation}\label{eq:dutc}
u_0 = v + \eps \theta, ~~~ d u_t = U_t d B_t + a_\nu(u_t) dt.
\end{equation}
Similarly to the above, we have $\lim_{t \to \infty} a_\nu(w_t) \sim \tau_{v + \eps \theta} \nu$, and therefore
\begin{align*}
\mathrm{W}_1(\tau_v \nu, \tau_{v + \eps \theta} \nu) & = \sup_{\tilde \varphi \in \mathrm{Lip}(\DC)} \left  |\EE \left [ \tilde \varphi \left ( \lim_{t \to \infty} a_\nu(w_t) \right ) \right ] - \EE \left [ \tilde \varphi \left ( \lim_{t \to \infty} a_\nu(u_t) \right ) \right ] \right  | \nonumber \\
& \leq  \EE \left  [ \left \| \lim_{t \to \infty} a_\nu(w_t) - \lim_{t \to \infty} a_\nu(u_t) \right  \|_1 \right ] \nonumber \\
& \leq \sqrt{n} \EE \left  [ \left | \lim_{t \to \infty} a_\nu(w_t) - \lim_{t \to \infty} a_\nu(u_t) \right  |_2 \right ].
\end{align*}
Consider the stopping time
$$
\tau := \inf\{t; ~~ u_t = w_t\}
$$  
and the event $E := \{\tau \leq 1\}$. By setting $U = \Id$ for all $t \geq \tau$, we get that $a_\nu(w_t) = a_\nu(u_t)$ for all $t \geq \tau$ almost surely, and therefore
$$
E \mbox{ holds } ~~ \Rightarrow ~~ \lim_{t \to \infty} |a_\nu(w_t)- a_\nu(u_t)| = 0.
$$
Moreover, since $a_\nu(w_t)$ is a martingale (as follows from part 2 of Proposition \ref{prop:basicprops}), we have
\begin{align*}
\left . \EE \left [ \left | a_\nu(w_t) - \lim_{s \to \infty} a_\nu(w_s) \right  | \right  | \FF_t \right ] & \leq \sqrt{ \EE \left . \left [ \left  | a_\nu(w_t) - \lim_{s \to \infty} a_\nu(w_s) \right  |^2 \right | \FF_t \right ]} \\
& = \sqrt{\mathrm{Tr}\bigl (\COV(\tau_{w_t} \nu ) \bigr)} = \sqrt{\mathrm{Tr}(A_\nu(w_t))},
\end{align*}
where the first equality uses the fact that $\lim_{s \to \infty} a_\nu(w_s) | \FF_t$ has the law $\tau_{w_t} \nu$, which follows from Proposition \ref{prop:alt}. Combining the above displays and using the triangle inequality, we conclude that for all $t \geq 1$,
\begin{align}
\mathrm{W}_1(\tau_v \nu, \tau_{v + \eps \theta} \nu) & \leq \sqrt{n} \EE \left  [ \mathbf{1}_{E^C} \left | \lim_{t \to \infty} a_\nu(w_t) - \lim_{t \to \infty} a_\nu(u_t) \right  |_2 \right ] \nonumber \\
& \leq \sqrt{n} \EE \Bigl [ | a_\nu(w_t) - a_\nu(u_t))| \Bigr ] + \EE \left [ \mathbf{1}_{E^C} \left (\sqrt{n\mathrm{Tr}(A_\nu(w_t))} + \sqrt{n\mathrm{Tr}(A_\nu(u_t))} \right ) \right ]. \label{eq:combinedbound}
\end{align}
We estimate every term on the right hand side separately, beginning with the first one. According to equation \eqref{eq:nablaa}, we have almost surely,
\begin{equation}\label{eq:contraction}
|a_\nu(u_t) - a_{\nu} (w_t)| \leq \beta |u_t - w_t|.
\end{equation}
Using equation \eqref{eq:dvtc} and \eqref{eq:dutc}, and by It\^o's formula, we have
\begin{align}
d |u_t - w_t|^2 & = 2 \langle u_t - w_t, d u_t - d w_t \rangle + \mathrm{Tr} \left ((U_t - \Id)^2\right ) dt \nonumber \\ 
& = 2 \langle u_t - w_t, (U_t  - \Id) d B_t \rangle + 2 \langle u_t - w_t, a_\nu(u_t) - a_{\nu} (w_t) \rangle dt + \mathrm{Tr} \left ((U_t - \Id)^2\right ) dt. \label{eq:ddist}
\end{align}
Set
$$
W_t = \int_0^t \left \langle \frac{u_t - w_t}{|u_t - w_t|}, d B_t \right \rangle, ~~  \forall t \leq \tau.
$$
Observe that $W_t$ is a standard Wiener process up to the time $\tau$. Finally, for all $t< \tau$, choose 
$$
U_t = \Id - 2 \left (\frac{w_t - u_t}{|w_t - u_t|} \right )^{\otimes 2},
$$
The reflection about the axis spanned by $w_t-u_t$. We have, by definition,
$$
\langle w_t - u_t, (\Id - U_t) d B_t \rangle = 2 |u_t - w_t| d W_t.
$$
Consequently, equation \eqref{eq:ddist} can be written
$$
d |u_t - w_t|^2 = 4 |u_t-w_t| d W_t + 4 dt + S_t dt
$$
where
$$
S_t = 2 \langle u_t - w_t, a_\nu(u_t) - a_{\nu} (w_t) \rangle \stackrel{\eqref{eq:contraction}}{\leq} 2 \beta |u_t - w_t|^2.
$$
Now define $Y_t = |u_t - w_t|$. Invoking It\^o's formula again, we have that up to time $\tau$, one has
\begin{align*}
d Y_t & = \frac{d \left (|u_t - w_t|^2 \right )}{2 |u_t - w_t|} - \frac{d \left [|u_t - w_t|^2 \right ]_t}{8  |u_t - w_t|^3} \\
& = 2 d W_t + \frac{S_t + 4}{2 |u_t - w_t|} dt - \frac{16 |u_t - w_t|^2 }{8 |u_t - w_t|^3} dt \\
& = 2 d W_t + \frac{S_t}{2 |u_t - w_t|} dt.
\end{align*}
Combining the two last displays, we learn that there exists an adapted process $Z_t$ such that
\begin{equation}\label{eq:dYt}
d Y_t = 2 d W_t + Z_t dt, ~~~ \forall t < \tau
\end{equation}
such that $Z_t \leq \beta Y_t$ almost surely for all $t \leq \tau$. A final application of It\^o's formula gives,
\begin{equation}\label{eq:ItoEy}
d (e^{-\beta t} Y_t) = e^{-\beta t} d Y_t - \beta e^{-\beta t} Y_t dt = 2 e^{- \beta t} d W_t + e^{- \beta t} (Z_t - \beta Y_t) dt.
\end{equation}
Therefore, $e^{-\beta t} Y_t$ is a supermartingale, and by the optional stopping theorem, we have
$$
\EE \left [Y_{t \wedge \tau} \right ] \leq e^{\beta t} \EE \left [e^{-\beta (t \wedge \tau)} Y_{t \wedge \tau} \right ] = e^{\beta t} Y_0 = e^{\beta t} \eps.
$$
It follows that
\begin{equation}\label{eq:firstterm}
\EE \Bigl [ | a_\nu(w_t) - a_\nu(u_t))|\Bigr ] \stackrel{\eqref{eq:contraction} }{\leq} \beta \EE \Bigl [ | w_t - u_t |\Bigr ] \leq \beta e^{\beta t} \eps.
\end{equation}
It remains to bound the second summand in the right hand side of \eqref{eq:combinedbound}. To this end, recall that $E$ is $\FF_1$-measurable, so we have for all $t \geq 1$,
\begin{equation}\label{eq:condtime1}
\EE \left [ \mathbf{1}_{E^C} \sqrt{n\mathrm{Tr}(A_\nu(w_t))} \right ] \leq \PP \left ( E^C \right ) \sup_{v \in \RR^n} \EE \left . \left [ \sqrt{n\mathrm{Tr}(A_\nu(w_t))} \right | w_1 = v \right ]. 
\end{equation}
To give an upper bound for the right hand side, we will first need the following lemma.
\begin{lemma}
Let $\eps > 0$. Let $W_t$ be a standard Brownian motion and let $Y_t, Z_t$ be adapted to $W_t$, which satisfy
$$
X_t = \eps + W_t + \int_0^t Z_s ds
$$
and such that $Z_t \leq 0$, almost surely for all $t$. Let $\tau = \inf\{t; ~ X_t = 0 \}$. One has
$$
\PP(\tau \geq s) \leq \frac{\eps}{\sqrt{s}}, ~~ \forall s > 0.
$$
\end{lemma}
\begin{proof}
By the reflection principle (see \cite[Theorem 2.19]{PM-BM}), we have 
\begin{align*}
\PP \left ( \min_{t \in [0,s] } X_t \leq - \eps \right ) & \geq \PP \left ( \min_{t \in [0,s] } |W_t + \eps| = 0 \right ) \\
 & = 2 \PP( W_s \leq - \eps ) \\
& = 1 - 2 (\PP( W_s \in [0, \eps] )  \\
& \geq 1 - \frac{2 \eps}{\sqrt{2 \pi s}}.
\end{align*}
\end{proof}
Let $t \to \TTT(t)$ be the unique increasing function satisfying
$$
4 \int_0^{\TTT(t)} e^{-2 \beta s} ds = t, ~~~ \forall 0 \leq t \leq 4 \int_0^{\infty} e^{-2 \beta s} ds.
$$
According to \eqref{eq:ItoEy}, 
$$
[e^{-\beta t} Y_t]_{\TTT(s)} = 4 \int_0^{\TTT(s)} e^{-2 \beta t} dt = s
$$
Since, as seen above, the process $e^{-\beta t} Y_t$ is a super-martingale, by applying a change of time we have that the process $t \to e^{-\beta \TTT(t)} Y_{\TTT(t)}$ is also a super-martingale (with respect to the filtration $\FF_{\TTT(t)}$). We may invoke the above lemma on this process and use the fact that $4 \int_0^1 e^{-2 \beta s} ds \geq \frac{1}{\beta}$, to attain
$$
\PP \left ( E^C \right ) = \PP(\tau \geq 1) \leq \eps \sqrt{\beta}.
$$
Finally, that by Equation \eqref{eq:Trdecreasing} we have for all $v \in \RR^n$ that
$$
\EE \left . \left [ \sqrt{n\mathrm{Tr}(A_\nu(w_t))} \right | w_1 = v \right ] \leq n e^{-(t-1)/16}.
$$
Combining the last two displays with equation \eqref{eq:condtime1}, we have
$$
\EE \left [ \mathbf{1}_{E^C} \sqrt{n\mathrm{Tr}(A_\nu(w_t))} \right ] \leq \eps \sqrt{\beta} n e^{-(t-1)/16}.
$$
By a similar argument, the same bound holds with $w_t$ replaced by $u_t$. Together with equations \eqref{eq:combinedbound} and \eqref{eq:firstterm} this gives
$$
\mathrm{W}_1(\tau_v \nu, \tau_{v + \eps \theta} \nu) \leq \eps \left (\beta \sqrt{n} e^{\beta t} + 2 \sqrt{\beta} n e^{-(t-1)/16} \right ), ~~~ \forall t \geq 1.
$$
Choosing $t = \frac{\log(2n)}{2 \beta + \frac{1}{8}}$ finally gives
$$
\mathrm{W}_1(\tau_v \nu, \tau_{v + \eps \theta} \nu) \leq 4 \eps \beta n^{1-1/(32 \beta)},
$$
completing the proof of Proposition \ref{prop:transport}.
\section{Proof of Fact \ref{fact:harmonic}} \label{sec:harmonic}
Let $\rho:[-1,1]^n \to \RR$ be the harmonic extension of $\frac{d \nu}{d \mu}$. A calculation gives
and write $g(x):=\log \rho (\tanh(x))$, 
\begin{align*}
g(x) & = \log \int_{\DC} \prod_{i \in [n]} \left (1 + \tanh(x_i) \right ) d \nu(x) \\
& = - \sum_{i \in [n]} \log \cosh(x_i) + \mathcal{L}[\nu](x).
\end{align*}
Therefore,
$$\COV(\tau_x\nu)=\nabla^2\mathcal{L}[\nu](x) = \nabla^2g(x)+\mathrm{diag}\left(\frac{1}{\cosh^2(x)}\right).$$
On the other hand, a direct calculate gives
$$\nabla^2_{i,j}g(x)=\frac{[\nabla^2_{i,j}\log \rho](\tanh(x))}{\cosh^2(x_i)\cosh^2(x_j)}-
\delta_{i,j}\frac{2\tanh(x_i)}{\cosh^2(x_i)}[\nabla_i\log \rho](\tanh(x)),$$
thus we have
$$\operatorname{Cov}_{i,j}(\tau_{x}\nu)=
\frac{[\nabla^2_{i,j}\log \rho](\tanh(x))}{\cosh^2(x_i)\cosh^2(x_j)}
-\delta_{i,j}\left(\frac{2\tanh(x_i)}{\cosh^2(x_i)}[\nabla_i\log \rho](\tanh(x))-\frac{1}{\cosh^2(x_i)}\right).$$
Observe that
$$\frac{1}{x_i-1}\le\nabla_i\log \rho(x)\le\frac{1}{x_i+1},$$
and therefore
$$\left\vert\frac{\nabla_i\log \rho(\tanh(x))}{\cosh^2(x_i)}\right\vert\le 2.$$
Finally, since $1/\cosh(x)\le 1$ and by assumption $\nabla^2\log \rho \preceq\beta\Id$, we obtain
$$\COV (\tau_{x}[\nu])\preceq(\beta+3)\Id.$$

\section{Proof of Theorem \ref{thm:entropy} }
Fix a measure $\nu$ on $\DC$ and consider the process $\nu_t$ constructed in Section \ref{sec:stochastic}. An application of \cite[Lemma 6]{Eldan-taming} gives
\begin{equation}\label{eq:ent1}
\mathcal{H}(\nu) = \frac{1}{2} \EE \left [\int_0^\infty \mathrm{Tr} \left (\COV(\nu_t) \right ) dt \right ].
\end{equation}
Define $h:\DC \to \RR$ by
$$
h(x) = - \sum_{i\in [n]} \frac{1+x_i}{2} \log \frac{1+x_i}{2} + \frac{1-x_i}{2} \log \frac{1-x_i}{2}.
$$
It is easily verified that 
\begin{equation}\label{eq:prodent}
\tilde{\mathcal{H}}(\nu) = h \left ( \int x d \nu(x) \right ) = h \left ( a_0 \right ).
\end{equation}
Recall (equation \eqref{eq:dat}) that
$$
d a_t = \COV(\nu_t) d B_t.
$$
Thus, by It\^o's formula,
$$
d h(a_t) = \langle \nabla h(a_t), d a_t \rangle + \frac{1}{2} \mathrm{Tr} \left ( \COV(\nu_t) \nabla^2 h(a_t) \COV(\nu_t)\right ) dt.
$$
A calculation gives
$$
\nabla^2 h(x) = - \mathrm{diag} \left ( \frac{1}{1-x_1^2}, \dots, \frac{1}{1-x_n^2} \right ).
$$
Observe also that since $\nu$ is supported on $\DC$,
$$
\mathrm{diag} \left ( \COV(\nu_t) \right ) = \Id - \mathrm{diag}(a_t)^2.
$$
Combining the last displays gives
\begin{equation}\label{eq:dht}
d h(a_t) = \frac{1}{2} \mathrm{Tr} \Bigl ( \COV(\nu_t) \left (\mathrm{diag} \left ( \COV(\nu_t) \right ) \right )^{-1} \COV(\nu_t)\Bigr ) dt + \mbox{martingale}.
\end{equation}
Using \eqref{eq:momentgenerating}, the condition \eqref{eq:semilc2} implies that, almost surely for all $t$, $\COV(\nu_t) \preceq \beta \mathrm{diag} \left ( \COV(\nu_t) \right )$ which yields
\begin{equation}\label{eq:dominate}
\mathrm{Tr} \Bigl ( \COV(\nu_t) \left (\mathrm{diag} \left ( \COV(\nu_t) \right ) \right )^{-1} \COV(\nu_t)\Bigr ) \leq \beta \mathrm{Tr}(\COV(\nu_t)).
\end{equation}
Combining the above finally gives
\begin{align*}
\tilde{\mathcal{H}}(\nu) & \stackrel{\eqref{eq:prodent}}{=} h(a_0) \\
& = h(a_0) - \EE \left [\lim_{t \to \infty} h(a_\infty) \right ] \\
& \stackrel{\eqref{eq:dht}}{=} \frac{1}{2} \EE \left [\int_0^\infty \mathrm{Tr} \Bigl ( \COV(\nu_t) \left (\mathrm{diag} \left ( \COV(\nu_t) \right ) \right )^{-1} \COV(\nu_t)\Bigr ) dt \right ] \\
& \stackrel{\eqref{eq:dominate}}{\leq} \frac{1}{2} \beta \EE \left [\int_0^\infty \mathrm{Tr} \Bigl ( \COV(\nu_t) \Bigr ) dt \right ] \\
& \stackrel{\eqref{eq:ent1}}{=} \beta \mathcal{H}(\nu).
\end{align*}
\bibliographystyle{plain}
\bibliography{bib}

\end{document}